\documentclass[11pt]{amsart}
\usepackage{fullpage,verbatim,amssymb,mathtools}
\usepackage{hyperref,subfig}
\usepackage[active]{srcltx}
\usepackage[usenames,dvipsnames]{color}
\usepackage[normalem]{ulem}

\makeatletter
\let\@@pmod\mod
\DeclareRobustCommand{\mod}{\@ifstar\@pmods\@@pmod}
\def\@pmods#1{\mkern4mu({\operator@font mod}\mkern 6mu#1)}
\makeatother

\definecolor{blue}{rgb}{0,0,1}
\definecolor{red}{rgb}{1,0,0}
\definecolor{green}{rgb}{0,.6,.2}
\definecolor{purple}{rgb}{1,0,1}

\long\def\red#1\endred{\textcolor{red}{#1}}
\long\def\blue#1\endblue{\textcolor{blue}{#1}}
\long\def\purple#1\endpurple{\textcolor{purple}{ #1}}
\long\def\green#1\endgreen{\textcolor{green}{#1}}

\newcommand{\sm}{\left(\begin{smallmatrix}}
\newcommand{\esm}{\end{smallmatrix}\right)}
\newcommand{\bpm}{\begin{pmatrix}}
\newcommand{\ebpm}{\end{pmatrix}}

\newcommand{\Z}{\mathbb{Z}}

\newcommand{\Q}{\mathbb{Q}}

\newcommand{\N}{\mathbb{N}}

\DeclareMathOperator{\ord}{ord}

\newtheorem{theorem}{Theorem}
\newtheorem{lemma}[theorem]{Lemma}
\newtheorem{proposition}[theorem]{Proposition}

\theoremstyle{remark}
\newtheorem*{remarks}{Remarks}
\numberwithin{theorem}{section}
\numberwithin{equation}{section}
\newcommand{\full}{{\rm full}}
\newcommand{\new}{{\rm new}}
\renewcommand{\min}{{\rm min}}
\newcommand{\n}{\star}

\begin{document}
\title{Dimensions of spaces of modular forms}
\author{Andrew R. Booker}
\email{andrew.booker@bristol.ac.uk}
\author{Min Lee}
\email{min.lee@bristol.ac.uk}
\address{School of Mathematics, University of Bristol,
Woodland Road, Bristol, BS8 1UG}
\begin{abstract}
We prove a conjecture of Ross concerning the value distribution of $\dim S_2^\new(\Gamma_0(N))$ for $N\in\N$, as well as analogous results for general weight $k\in2\N$ and the full and twist-minimal spaces $S_k(\Gamma_0(N))$, $S_k^\min(\Gamma_0(N))$.
\end{abstract}
\maketitle

\section{Introduction}
In \cite[Conjecture~27]{Martin}, Martin conjectured that every non-negative integer can be expressed as $\dim S_2^\new(\Gamma_0(N))$ for some $N\in\N$. 
Recently, Ross \cite{Ross1} disproved this conjecture, showing that dimension $67846$ is not attained, and made the counter-conjecture \cite[Conjecture~6.1]{Ross2} that the set of dimensions has density zero in the non-negative integers. In this paper we prove a general form of Ross' conjecture that applies to all weights $k\in2\N$ and includes the full and twist-minimal spaces, $S_k(\Gamma_0(N))$ and $S_k^\min(\Gamma_0(N))$.
Our main tool is the value distribution of the Euler totient and similar multiplicative functions, whose study was begun by Pillai \cite{Pillai} and Erd\H{o}s \cite{Erdos}, and perfected by Ford \cite{Ford1,Ford2}.

For $k\in2\N$, let
\begin{align*}
d_k^\full(N)=\dim S_k(\Gamma_0(N)),\quad
d_k^\new(N)=\dim S_k^\new(\Gamma_0(N)),\quad
d_k^\min(N)=\dim S_k^\min(\Gamma_0(N))
\end{align*}
and set
\[
D_k^\n(x)=\#\left\{d_k^\n(N):N\in\N,\,d_k^\n(N)\le\frac{k-1}{12}x\right\}
\quad\text{for }\n\in\{\full,\new,\min\}.
\]
Our precise result is the following.
\begin{theorem}\label{thm:main}
Uniformly for $k\in2\N$ and $\n\in\{\full,\new,\min\}$, we have
\[
D_k^\n(x)=\frac{x}{\log{x}}
\exp\!\left(C\log^2\!\left(\frac{\log\log{x}}{\log\log\log{x}}\right)+O\bigl(\log\log\log{x}\bigr)\right)
\quad\text{for }x\ge16,
\]
where $C=0.8178146\ldots$ is the constant defined in \cite[(1.5)]{Ford2}.
\end{theorem}

\begin{remarks}\
\begin{enumerate}
\item If one instead fixes $N$ and varies $k\in2\N$ then Ross \cite[Theorem~1.3]{Ross2} showed that $S_k(\Gamma_0(N))$ and $S_k^\new(\Gamma_0(N))$ do attain every dimension for some small values of $N$. The proof extends easily to $S_k^\min(\Gamma_0(N))$, and more generally one can see that $\bigl\{d_k^\n(N):k\in2\N\bigr\}$ has positive density for every $N$.
\item With appropriate modifications our proof could be adapted to the spaces $S_k^\n(\Gamma_0(N),\chi)$, where $k\ge2$ is fixed, $\chi\pmod{q}$ is a fixed primitive character satisfying $\chi(-1)=(-1)^k$, and $N$ varies over $q\N$.

For $k=1$ the question is much more subtle, and for some $\chi$ of small conductor it seems plausible that $\bigl\{\dim S_1^\n(\Gamma_0(N),\chi):N\in q\N\bigr\}$
contains every non-negative integer. These spaces are expected to be dominated by dihedral forms, so one is led to study the value distribution of class numbers $h(\Delta)$ for fundamental discriminants $\Delta<0$. Conjectures in \cite{Soundararajan,HJKMP} suggest that a given class number $h$ occurs with multiplicity $\gg\frac{h}{\log{h}}$.
\end{enumerate}
\end{remarks}

\subsection*{Acknowledgements}
We thank Kevin Ford for clarifying some points about \cite[Theorem~14]{Ford2}.

\section{Dimension formulae}
Let us first recall the dimension formulae for $S_k(\Gamma_0(N))$ and $S_k^\new(\Gamma_0(N))$,
as computed by Martin \cite{Martin}.
For $k\in2\N$, define 
\[
c_2(k) = -\frac14\left(\frac{-4}{k-1}\right),\quad
c_3(k) = -\frac13\left(\frac{-3}{k-1}\right),\quad
\text{and}\quad
\delta_2(k)=\begin{cases}
1&\text{if }k=2,\\
0&\text{if }k\ne2.
\end{cases}
\]
\begin{proposition}[\cite{Martin}, Proposition~12]
\label{prop:dimfull}
For $N\in\N$ and $k\in2\N$, we have
\[
\dim S_k(\Gamma_0(N))=
\frac{k-1}{12}\psi^\full(N)-\frac12\nu_\infty^\full(N)
+c_2(k)\nu_2^\full(N)+c_3(k)\nu_3^\full(N)+
\delta_2(k),
\]
where $\psi^\full$, $\nu_\infty^\full$, $\nu_2^\full$, $\nu_3^\full$ are multiplicative functions given on prime powers $p^e>1$ by
\[
\psi^\full(p^e)=p^e+p^{e-1},\quad
\nu_\infty^\full(p^e)=\begin{cases}
2p^{\frac{e-1}{2}}&\text{if }2\nmid e,\\
p^{\frac{e}{2}}+p^{\frac{e}{2}-1}&\text{if }2\mid e,
\end{cases}
\]
\[
\nu_2^\full(p^e)=\begin{cases}
1&\text{if }p^e=2,\\
2&\text{if }p\equiv1\bmod4,\\
0&\text{otherwise},
\end{cases}
\quad\text{and}\quad
\nu_3^\full(p^e)=\begin{cases}
1&\text{if }p^e=3,\\
2&\text{if }p\equiv1\bmod3,\\
0&\text{otherwise}.
\end{cases}
\]
\end{proposition}
\begin{proposition}[\cite{Martin}, Theorem~1]
\label{prop:dimnew}
For $N\in\N$ and $k\in2\N$, we have
\[
\dim S_k^\new(\Gamma_0(N))=
\frac{k-1}{12}\psi^\new(N)-\frac12\nu_\infty^\new(N)
+c_2(k)\nu_2^\new(N)+c_3(k)\nu_3^\new(N)+
\delta_2(k)\mu(N),
\]
where $\psi^\new$, $\nu_\infty^\new$, $\nu_2^\new$, $\nu_3^\new$ are multiplicative functions given on prime powers $p^e>1$ by
\[
\psi^\new(p^e)=\begin{cases}
p-1&\text{if }e=1,\\
p^2-p-1&\text{if }e=2,\\
p^{e-3}(p-1)^2(p+1)&\text{if }e>2,
\end{cases}
\quad
\nu_\infty^\new(p^e)=\begin{cases}
p-2&\text{if }e=2,\\
p^{\frac{e}{2}-2}(p-1)^2&\text{if }2\mid e>2,\\
0&\text{otherwise}
\end{cases}
\]
\[
\nu_2^\new(p^e)=\begin{cases}
-2&\text{if }p\equiv-1\bmod4\text{ and }e=1,\\
1&\text{if }p\equiv-1\bmod4\text{ and }e=2\text{ or }p^e=8,\\
-1&\text{if }p\equiv1\bmod4\text{ and }e=2\text{ or }p^e\in\{2,4\},\\
0&\text{otherwise},
\end{cases}
\]
and
\[
\nu_3^\new(p^e)=\begin{cases}
-2&\text{if }p\equiv-1\bmod3\text{ and }e=1,\\
1&\text{if }p\equiv-1\bmod3\text{ and }e=2\text{ or }p^e=27,\\
-1&\text{if }p\equiv1\bmod3\text{ and }e=2\text{ or }p^e\in\{3,9\},\\
0&\text{otherwise},
\end{cases}
\]
\end{proposition}

The twist-minimal space $S_k^\min(\Gamma_0(N))$ is the subspace of $S_k^\new(\Gamma_0(N))$ spanned by newforms that cannot be obtained by twisting a lower-level newform by a Dirichlet character.
We compute the dimension of $S_k^\min(\Gamma_0(N))$ using the trace formula derived by Child \cite{Child}.
\begin{theorem}\label{thm:dimmin}
For $N\in\N$ and $k\in2\N$, we have
\[
\dim S_k^\min(\Gamma_0(N))=
\frac{k-1}{12}\psi^\min(N)-\frac12\nu_\infty^\min(N)
+c_2(k)\nu_2^\min(N)+c_3(k)\nu_3^\min(N)+
\delta_2(k)\mu(N).
\]
where $\psi^\min$, $\nu_\infty^{\min}$, $\nu_2^\min$, and $\nu_3^\min$ are multiplicative functions given on prime powers $p^e>1$ by
\[
\psi^\min(p^e)=
\frac{p-1}{(2,p-1,e)}
\begin{cases}
1 & \text{if }e=1,\\
p-1 & \text{if }e=2, \\
p^{e-3}(p^2-1) & \text{if }e>2,
\end{cases}
\qquad
\nu_\infty^\min(p^e)=\begin{cases}
p^{\frac{e}{2}-2}&\text{if }p=2\text{ and }2\mid e>2,\\
0&\text{otherwise},
\end{cases}
\]
\[
\nu_2^\min(p^e) = \begin{cases}
-2 \mu(p^{e-1}) & \text{if } p\equiv-1\bmod{4},\\
-1&\text{if }p^e\in\{2,4\},\\
1&\text{if }p^e=8,\\
0 & \text{otherwise},
\end{cases}
\]
and
\[
\nu_3^\min(p^e) = \begin{cases}
-2\mu(p^{e-1})&\text{if }p\equiv-1\bmod{3} 
\text{ and }p^e\ne4, \\
-1&\text{if }p^e\in\{3,9\},\\
1&\text{if }p^e\in\{4,27\},\\
0 & \text{otherwise.}
\end{cases}
\]
\end{theorem}

\begin{proof} 
Taking $n=1$ in \cite[Theorem 2.1]{Child}, we obtain
\[
\dim S_k^{\min}(N) = C_1-C_2-C_3+C_4,
\]
where $C_1=\frac{k-1}{12}\psi^\min(N)$, $C_4=\delta_2(k)\mu(N)$, and the other terms are as follows.

For $C_2$, we have
\[
C_2 = \sum_{\substack{t\in\Z,\,d=t^2-4<0\\\rho^2-t\rho+1=0,\,\Im\rho>0}} \frac{\rho^{k-1}-\bar{\rho}^{k-1}}{\rho-\bar{\rho}} \frac{h(d)}{w(d)} \prod_{p\mid N} S_p^\min\bigl(p^{\ord_p(N)}, 1, t, 1\bigr),
\]
where $h(d)$ is the class number of $\Q(\sqrt{d})$, $w(d)$ is its number of roots of unity, and $S_p^\min$ is a multiplicative function defined in \cite[(2.12)--(2.17)]{Child}. Note that the sum has only the three terms $t=0,\pm1$. 
We have
\[
-\frac{\rho^{k-1}-\bar{\rho}^{k-1}}{\rho-\bar{\rho}}
= 2(-1)^{\frac{k}{2}} \frac{\cos((k-1)\phi_t)}{\sqrt{4-t^2}} \quad\text{where }\phi_t=\arcsin\!\left(\frac{t}{2}\right) \in \left[-\frac{\pi}{2}, \frac{\pi}{2}\right].
\] 
For $t=0$, this yields
\[
-\frac{\rho^{k-1}-\bar{\rho}^{k-1}}{\rho-\bar{\rho}}\frac{h(d)}{w(d)} = \frac{(-1)^{\frac{k}{2}}}{4}
= c_2(k),
\]
and for $t=\pm1$,
\begin{align*}
-2\frac{\rho^{k-1}-\bar{\rho}^{k-1}}{\rho-\bar{\rho}}\frac{h(d)}{w(d)}
&= \frac{(-1)^{\frac{k}{2}}}{3} 
\begin{cases}
(-1)^{\frac{k}{6}}  & \text{ if } k\equiv 0\bmod{6}, \\
(-1)^{\frac{k-2}{6}}  & \text{ if } k\equiv 2\bmod{6}, \\
0 & \text{ if } k\equiv 4\bmod{6}
\end{cases}\\
&=c_3(k),
\end{align*}
so that
\[
-C_2 = c_2(k)\prod_{p\mid N} S_p^\min\bigl(p^{\ord_p(N)}, 1, 0, 1\bigr) 
+ c_3(k)\prod_{p\mid N} S_p^\min\bigl(p^{\ord_p(N)}, 1, \pm1, 1\bigr).
\]
From \cite[(2.12)]{Child} we see that
\[
S_p^{\min}(p^e, 1, t, 1)
= \left(\left(\frac{d}{p}\right)-1\right) \mu(p^{e-1})
\quad\text{when }2<p\nmid d,
\]
and
\[
S_p^{\min}(p^e, 1, t, 1)
= \begin{cases}
-1 & \text{if }e\le2,\\
1 & \text{if }e=3,\\
0 & \text{otherwise}
\end{cases}
\quad\text{when }p=3,\,t=\pm1.
\]
For $p=2$, we see from \cite[(2.14), (2.15)]{Child} that
\[
S_p^{\min}(p^e, 1, 0, 1) = \begin{cases}
-1 & \text{if }e\le2, \\
1 & \text{if }e=3, \\
0 & \text{otherwise} 
\end{cases}
\qquad\text{and}\qquad
S_p^{\min}(p^e, 1, \pm1, 1) = \begin{cases}
-2 & \text{if } e=1, \\
1 & \text{if } e=2, \\
0 & \text{otherwise}.
\end{cases}
\]
In all cases these match the local factors in $\nu_2^\min$ and $\nu_3^\min$.

Finally, we have $C_3=0$ unless $N=1$ or $N=2^e$ with $2\mid e$ and $e>2$. In the latter case, we have
$2C_3 = 2^{\frac{e}{2}-2}$, matching the local factor of $\nu_\infty^\min$.
\end{proof}

\section{Proof of Theorem~\ref{thm:main}}
We begin with a few lemmas.
\begin{lemma}\label{lem:nubounds}
For $N\in\N$ and $\n\in\{\full,\new,\min\}$, we have 
\[
0\le\nu_\infty^\n(N)\le\frac{\psi^\n(N)}{\sqrt{N}}
\quad\text{and}\quad
\bigl|c_2(k)\nu_2^\n(N)+c_3(k)\nu_3^\n(N)\bigr|
\le \frac{7}{12}2^{\omega(N)}
\]
\end{lemma}
\begin{proof}
Define $f^\n(N)=\nu_\infty^\n(N)\sqrt{N}/\psi^\n(N)$. Then for prime powers $p^e>1$ we compute that
\[
f^\full(p^e)=\begin{cases}
\frac{2}{p^{\frac12}+p^{-\frac12}}&\text{if }2\nmid e,\\
1&\text{if }2\mid e,
\end{cases}
\qquad
f^\new(p^e)=\begin{cases}
\frac{p^2-2p}{p^2-p-1}&\text{if }e=2,\\
\frac{p}{p+1}&\text{if }2\mid e>2,\\
0&\text{otherwise},
\end{cases}
\]
and
\[
f^\min(p^e)=\begin{cases}
\frac23&\text{if }p=2\text{ and }2\mid e>2,\\
0&\text{otherwise}.
\end{cases}
\]
Thus the local factors are always non-negative and bounded by $1$, which proves the first inequality.

For the second, we have $|\nu_2^\n(N)|, |\nu_3^\n(N)| \leq 2^{\omega(N)}$, $|c_2(k)|\le\frac14$, $|c_3(k)|\le\frac13$.
\end{proof}

We recall that a number $N\in\N$ is called \emph{squarefull} if $p^2\mid N$ for every prime $p\mid N$. Let $H(N)$ denote the squarefull part of $N$, i.e.\ the largest squarefull number dividing $N$. Note that $N/H(N)$ is squarefree and $(H(N),N/H(N))=1$.
\begin{lemma}\label{lem:squarefull}
We have
\[
N^{1+\varepsilon}\gg_\varepsilon
\psi^\full(N)\ge\psi^\new(N)\ge\psi^\min(N)
\gg_\varepsilon N^{1-\varepsilon}
\quad\text{for }N\in\N,\,\varepsilon>0
\]
and
\[
\sum_{\substack{N>x\\N\text{ squarefull}}}\frac1{\psi^\n(N)}
\ll\frac{\log{x}}{\sqrt{x}}
\quad\text{for }x\ge2\text{ and }
\n\in\{\full,\new,\min\}.
\]
\end{lemma}

\begin{proof}
We trivially have $d_k^\full(N)\ge d_k^\new(N)\ge d_k^\min(N)$, and multiplying by $\frac{12}{k-1}$ and taking $k\to\infty$ we deduce that $\psi^\full(N)\ge\psi^\new(N)\ge\psi^\min(N)$.
Note that
\[
\psi^\full(N)=N\prod_{p\mid N}\left(1+\frac1p\right)
\ll N\log\log(3N),
\]
which establishes the upper bound.

Next let $f(N)=\frac{\varphi(N)^2}{2^{\omega(N)}N}$. Then for a prime power $p^e>1$, we have
\[
\frac{\psi^\min(p^e)}{f(p^e)}
=\begin{cases}
\frac{2p}{p-1}&\text{if }e=1,\\
\frac{2}{(2,p-1)}&\text{if }e=2,\\
\frac{2(p+1)}{(2,p-1,e)p}&\text{if }e>2.
\end{cases}
\]
This is at least $1$ in all cases, so we have
\[
\psi^\min(N)\ge f(N)
\ge N^{1-\frac{\log2+o(1)}{\log\log(3N)}}
\quad\text{as }N\to\infty,
\]
which establishes the lower bound.
Using this and the inequality $\frac{2^{\omega(ab)}}{\varphi(ab)^2}\le\frac{2^{\omega(a)}}{\varphi(a)^2}\frac{2^{\omega(b)}}{\varphi(b)^2}$, we have 
\begin{align*} 
\sum_{\substack{N\text{ squarefull}\\N>x}} \frac{1}{\psi^\n(N)}
& \leq \sum_{\substack{N\text{ squarefull}\\N>x}} \frac1{f(N)}
= \sum_{\substack{a, b\in \N\\ a^2b^3>x}}\frac{\mu^2(b)}{f(a^2b^3)}
= \sum_{\substack{a, b\in \N\\ a^2b^3>x}} \frac{ \mu^2(b)2^{\omega(ab)}}{b\varphi(ab)^2} 
\\ & \leq \sum_{\substack{a, b\in \N\\ a^2b^3>x}} \frac{2^{\omega(a)}}{\varphi(a)^2} \frac{2^{\omega(b)}}{b\varphi(b)^2}
= \sum_{b\leq x^{\frac{1}{3}}} \frac{2^{\omega(b)}}{b\varphi(b)^2} \sum_{a>\sqrt{\frac{x}{b^3}}} \frac{2^{\omega(a)}}{\varphi(a)^2}
+ \sum_{b>x^{\frac{1}{3}}}\frac{2^{\omega(b)}}{b\varphi(b)^2} 
\sum_{a\in \N} \frac{2^{\omega(a)}}{\varphi(a)^2}. 
\end{align*}
Since $\sum_{a\in \N} \frac{2^{\omega(a)}}{\varphi(a)^2}$ converges and 
$\sum_{b>x^{\frac{1}{3}}} \frac{2^{\omega(b)}}{b\varphi(b)^2} \ll_\varepsilon x^{-\frac23+\varepsilon}$,
we have
\[\sum_{b>x^{\frac{1}{3}}}\frac{2^{\omega(b)}}{b\varphi(b)^2} 
\sum_{a\in \N} \frac{2^{\omega(a)}}{\varphi(a)^2} \ll_\varepsilon x^{-\frac23+\varepsilon}.\]

To estimate the inner sum when $b\le x^{\frac13}$, let
\begin{align*} 
F(s) & = \sum_{n=1}^\infty \frac{2^{\omega(n)}n^2}{\varphi(n)^2}\frac{1}{n^s}
= \prod_p \bigg(1+\frac{2}{(1-p^{-1})^2}\sum_{j=1}^\infty p^{-js}\bigg)
\\ & = \zeta(s)^2 \prod_p \left(1+ \frac{2p(2p-1)}{(p-1)^2}p^{-s-1} + \frac{p^2+2p-1}{(p-1)^2}p^{-2s}\right).
\end{align*}
The product over $p$ converges absolutely for $\Re(s)>\frac12$, so $F(s)$ continues analytically to $\Re(s)>\frac12$ apart from a double pole at $s=1$. 
Applying \cite[Theorem 3.1]{Kato}, we have 
\[S(x)\coloneq \sum_{n\leq x} \frac{2^{\omega(n)}n^2}{\varphi(n)^2} \ll x(1+\log x)
\quad\text{for }x\ge1,\]
which yields
\[\sum_{n>x} \frac{2^{\omega(a)}}{\varphi(a)^2}
=\int_x^\infty t^{-2}\,dS(t)
\le 2\int_{x}^\infty S(t) t^{-3} \, dt
\ll \int_x^\infty \frac{1+\log t}{t^2}\, dt = \frac{2+\log{x}}{x}.\]
Thus for $x\ge2$ we have 
\begin{align*}
\sum_{\substack{N>x\\N\text{ squarefull}}} \frac{1}{\psi^\n(N)} \ll_\varepsilon
x^{-\frac23+\varepsilon}+\frac{\log{x}}{\sqrt{x}} \sum_{b\leq x^{\frac{1}{3}}} \frac{2^{\omega(b)}}{\sqrt{b}\varphi(b)^2}
\ll \frac{\log x}{\sqrt{x}}.
\end{align*}
\end{proof}

\begin{lemma}\label{lem:eta}
Let $\eta = \zeta(\frac32)/\zeta(3)=2.17325\ldots$. Then
\[
\#\bigl\{N\text{ squarefull}:N\le x\bigr\}\le\eta\sqrt{x}
\quad\text{and}\quad
\sum_{\substack{N\text{ squarefull}\\N>x}}\frac1N
\le\frac{2\eta}{\sqrt{x}}.
\]
\end{lemma}
\begin{proof}
The first estimate is \cite[(8)]{Golomb}, and the second follows from the first by partial summation.
\end{proof}

\begin{lemma}\label{lem:values}
For any $k\in2\N$, $r,s\in\N$ and $\n\in\{\new,\min\}$, we have
\[
\#\left\{d_k^\n(N)-\frac{k-1}{12}\psi^\n(N):
N\in\N,\,\omega(N)<r,\,\sqrt{N}\notin\N\right\}
\le3(2r+1)^2
\]
and
\[
\#\left\{d_k^\full(N)-\frac{k-1}{12}\psi^\full(N):
N\in\N,\,\omega(N)<r,\,H(N)\le s\right\}
\le\eta\sqrt{s}r(r+1)^2.
\]
\end{lemma}
\begin{proof}
For the full space, consider $N=N_1N_2$, where $N_1$ is squarefree, $N_2\le s$ is squarefull, and $(N_1,N_2)=1$. Then
\begin{equation}\label{eq:fulldiff}
\begin{aligned}
d_k^\full&(N_1N_2)-\frac{k-1}{12}\psi^\full(N_1N_2)\\
&=-\frac12\nu_\infty^\full(N_1)\nu_\infty^\full(N_2)
+c_2(k)\nu_2^\full(N_1)\nu_2^\full(N_2)
+c_3(k)\nu_3^\full(N_1)\nu_3^\full(N_2)+\delta_2(k).
\end{aligned}
\end{equation}
In view of Proposition~\ref{prop:dimfull}, when $\omega(N)<r$ there are at most $r$ possible values of $\nu_\infty^\full(N_1)$, and at most $r+1$ possible values of $\nu_2^\full(N_1)$ and $\nu_3^\full(N_1)$. By Lemma~\ref{lem:eta}, there are at most $\eta\sqrt{s}$ choices for $N_2$ when $H(N)\le s$.
This yields at most $\eta\sqrt{s}r(r+1)^2$ possibilities for the right-hand side of \eqref{eq:fulldiff}.

Similarly, for $\n\in\{\new,\min\}$ we have
\[
d_k^\n(N)-\frac{k-1}{12}\psi^\n(N)
=-\frac12\nu_\infty^\n(N)
+c_2(k)\nu_2^\n(N)
+c_3(k)\nu_3^\n(N)+\delta_2(k)\mu(N),
\]
and for $N$ with $\omega(N)<r$ and $\sqrt{N}\notin\N$, we have $\nu_\infty^\n(N)=0$ and there are at most $2r+1$ possibilities for $\nu_2^\n(N)$ and $\nu_3^\n(N)$, and at most three possibilities for $\mu(N)$.
\end{proof}

\begin{lemma}\label{lem:exceptions}
For $\n\in\{\full,\new,\min\}$ and $x>1$,
\[
\#\left\{N\in\N:\operatorname{min}\left\{\psi^\n(N),\frac{12}{k-1}d_k^\n(N)\right\}\le x
\text{ and }\Bigl(\omega(N)>3\log\log{x}\text{ or }\sqrt{N}\in\N\Bigr)\right\}
\ll\frac{x}{\log{x}}.
\]
\end{lemma}
\begin{proof}
From Lemmas~\ref{lem:nubounds} and \ref{lem:squarefull}, we have $d_k^\n(N) = \frac{k-1}{12} \psi^\n(N)+O(N^{\frac{1}{2}+\varepsilon})$, 
and thus
\[
\min\left\{\psi^\n(N), \frac{12}{k-1}d_k^\n (N)\right\}\leq x \implies
\psi^\n(N) \leq x+ O\bigl(x^{\frac{1}{2}+\varepsilon}\bigr).\]
Write $N = N_1N_2$ with $N_1$ squarefree, $N_2$ squarefull, and $(N_1, N_2)=1$.
Since $\psi^\n$ is multiplicative and $N_1$ is squarefree, for $x\ge3$ and a suitable constant $A>0$, we have 
\[\varphi(N_1)\le\psi^\n(N_1) \le \frac{x+O\bigl(x^{\frac12+\varepsilon}\bigr)}{\psi^\n(N_2)}
\implies N_1\le\frac{Ax\log\log{x}}{\psi^\n(N_2)}.\]

We first count the number of $N$ with $N_2 > \log^3 x$.
By Lemma \ref{lem:squarefull} we have 
\begin{align*}
\#\left\{N=N_1N_2 \in \N:\, N_1 \le \frac{Ax\log\log{x}}{\psi^\n(N_2)} \text{ and } N_2 > \log^3x \right\}
&\le \sum_{N_2>\log^3x}\frac{Ax\log\log x}{\psi^\n(N_2)}\\
&\ll \frac{x(\log\log x)^2 }{(\log x)^{\frac{3}{2}}},
\end{align*}
so these make a negligible contribution.
Next note that if $\sqrt{N}\in\N$ then $N=N_2$, so the number of such $N$ with $N_2\le\log^3{x}$ at most $(\log{x})^{\frac32}$, which is again negligible.

Finally, suppose $\omega(N)>3\log\log{x}$ and $N_2\leq \log^3{x}$. Then $\omega(N_2) \ll \frac{\log\log x}{\log\log \log x}$, so for sufficiently large $x$ we have 
\[\omega(N_1) > 3\log\log{x}-\omega(N_2) > 2.9\log\log\left(\frac{Ax\log\log{x}}{\psi^\n(N_2)}\right). \]
Applying \cite[Lemma 2.2]{Ford2} and Lemma~\ref{lem:squarefull}, we have
\begin{align*}
\sum_{N_2\le\log^3{x}}
\#\left\{N_1\in\N:N_1\le\frac{Ax\log\log x}{\psi^\n(N_2)} \text{ and } \omega(N_1) > 2.9\log\log\left(\frac{Ax\log\log{x}}{\psi^\n(N_2)}\right)\right\}\\
\ll\frac{x(\log\log{x})^2}{(\log{x})^{2.9\log{2}-1}}\sum_{N_2\le\log^3{x}}\frac1{\psi^\n(N_2)}
\ll\frac{x(\log\log{x})^2}{(\log{x})^{2.9\log{2}-1}}.
\end{align*}
Since $2.9\log{2}-1=1.01012\ldots>1$, this is $O\bigl(\frac{x}{\log{x}}\bigr)$, as claimed.
\end{proof}

\begin{proposition}\label{prop:ford}
Let
\[
V_{\psi^\n}(x)=\#\bigl\{\psi^\n(N):N\in\N,\,\psi^\n(N)\le x\bigr\}
\quad\text{for }\n\in\{\full,\new,\min\}
\]
and
\[
\rho(x)=\frac{1}{\log{x}}
\exp\!\left(C\log^2\!\left(\frac{\log\log{x}}{\log\log\log{x}}\right)
+D\log\log\log{x}+(D+\tfrac12-2C)\log\log\log\log{x}
\right),
\]
where $C$ and $D$ are as defined in \cite[(1.5) and (1.6)]{Ford2}. Then
\[
V_{\psi^\n}(x)\asymp x\rho(x)
\quad\text{for }x\ge16.
\]
\end{proposition}
\begin{proof}
This follows from \cite[Theorem~14]{Ford2}. To verify the hypotheses, note that $\{\psi^\n(p)-p:p\text{ prime}\}$ is a singleton set (either $\{1\}$ or $\{-1\}$) not containing $0$, and that $\sum_{N\text{ squarefull}}\frac{N^\delta}{\psi^\n(N)}$ converges for any $\delta<\frac12$, by Lemma~\ref{lem:squarefull}.
\end{proof}

With these ingredients in place, we may complete the proof of Theorem~\ref{thm:main}. We begin with $\n\in\{\new,\min\}$, which are a bit easier since $\nu_\infty^\n(N)=0$ when $\sqrt{N}\notin\N$.

Let $x>0$ be a large real number and consider $N\in\N$ such that $\sqrt{N}\notin\N$, $\omega(N)\le3\log\log{x}$, and $d_k^\n(N)\le\frac{k-1}{12}x$. From Lemmas~\ref{lem:squarefull} and \ref{lem:nubounds} we see that
\[
\Delta\coloneq\frac{12}{k-1}d_k^\n(N)-\psi^\n(N)
\ll_\varepsilon\frac{x^{\frac12+\varepsilon}}{k},
\]
so for large enough $x$ we have $|\Delta|\le\frac{x}{2}$. Moreover, Lemma~\ref{lem:values} implies that $\Delta$ assumes $O((\log\log{x})^2)$ values as $N$ varies, with an implied constant that is independent of $k$. Adding in the contribution from Lemma~\ref{lem:exceptions}, we therefore have
\[
D_k^\n(x)\ll (\log\log{x})^2 V_{\psi^\n}(3x/2)+\frac{x}{\log{x}}
\]
and
\[
V_{\psi^\n}(x/2)\ll(\log\log{x})^2 D_k^\n(x)+\frac{x}{\log{x}}.
\]
In view of Proposition~\ref{prop:ford}, it follows that
\[
D_k^\n(x)=
\frac{x}{\log{x}}
\exp\!\left(C\log^2\!\left(\frac{\log\log{x}}{\log\log\log{x}}\right)+O\bigl(\log\log\log{x}\bigr)\right)
\]
for all sufficiently large $x$. Finally, note that $D_k(x)\ge1$ for $x\ge\frac{12}{11}$, so we can take the implied constant large enough (and uniform in $k$) to cover all $x\ge16$.

For the full space, first note that $\nu_\infty^\full(N)\ge N$, so by Lemma~\ref{lem:nubounds}, we have
\[
d_k^\full(N)\ge\psi^\full(N)\left(\frac{k-1}{12}-\frac1{2\sqrt{N}}\right)-\frac{7}{12}2^{\omega(N)}
\ge\frac{k-1}{12}N-O\bigl(\sqrt{N}\bigr).
\]
Therefore, if $d_k^\full(N)\le\frac{k-1}{12}x$ then $N\le x+O(\sqrt{x})$. Now the idea is to write $N=N_1N_2$, where $N_1$ is squarefree, $N_2$ is squarefull, and $(N_1,N_2)=1$.
The total number of $N=N_1N_2\le x+O(\sqrt{x})$ with $N_2>\log^2{x}$ is at most
\[
\sum_{\substack{N_2\text{ squarefull}\\N_2>\log^2{x}}}
\frac{x+O(\sqrt{x})}{N_2}
\ll\frac{x}{\log{x}}.
\]
For fixed $N_2\le\log^2{x}$ we apply the preceding argument (now with factors of $(\log\log{x})^3$ to account for the higher power of $r$ in Lemma~\ref{lem:values}) to prove
\begin{align*}
\#&\left\{d_k^\full(N_1N_2):N_1\text{ squarefree},\,(N_1,N_2)=1,\,\omega(N_1N_2)\le3\log\log{x},\,d_k^\full(N_1N_2)\le\frac{k-1}{12}x\right\}\\
&\ll\frac{x}{N_2}\rho\!\left(\frac{x}{N_2}\right)
(\log\log{x})^3
=\frac{x\rho(x)}{N_2}\exp(O(\log\log\log{x})).
\end{align*}
Summing over $N_2\le\log^2{x}$ and adding the contributions from $N_2>\log^2{x}$ and $N>3\log\log{x}$ gives the upper bound.

For the lower bound, we could take $N_2=1$ and use an estimate for the value distribution of $\psi^\full(N)$ restricted to squarefree $N$. Although it is not stated outright in \cite[Theorem~14]{Ford2}, the proof of the lower bound requires only squarefree $N$. However, we can circumvent this assumption and rely only on the stated result using sufficiently large bounded values of $N_2$, as follows.

For $s\in\N$, we wish to derive a lower estimate for
\[
V_{\psi^\full}^s(x)=\#\bigl\{\psi^\full(N):N\in\N,\,H(N)\le s,\,\psi^\full(N)\le x\bigr\}.
\]
By Proposition~\ref{prop:ford}, there are constants $\alpha,\beta>0$ such that
\[
\alpha x\rho(x)\le V_{\psi^\full}(x)\le\beta x\rho(x)
\quad\text{for }x\ge16.
\]
Considering $N=N_1N_2$, for large $x$ we have
\begin{align*}
V_{\psi^\full}^s(x)&\ge V_{\psi^\full}(x)-\sum_{\substack{N_2\text{ squarefull}\\s<N_2\le\log^2{x}}}
V_{\psi^\full}\!\left(\frac{x}{N_2}\right)
-\sum_{\substack{N_2\text{ squarefull}\\N_2>\log^2{x}}}\frac{x}{N_2}\\
&\ge\alpha x\rho(x)-\sum_{\substack{N_2\text{ squarefull}\\s<N_2\le\log^2{x}}}
\beta\frac{x}{N_2}\rho\!\left(\frac{x}{N_2}\right)
-\frac{2\eta x}{\log{x}}\\
&\ge\alpha x\rho(x)-\frac{2\eta\beta}{\sqrt{s}}x\rho(x)\left(1+O\!\left(\frac{\log\log{x}}{\log{x}}\right)\right)-\frac{2\eta x}{\log{x}}.
\end{align*}
Choosing $s>(4\eta\beta/\alpha)^2$, this is at least $\frac12\alpha x\rho(x)$ for sufficiently large $x$.

Finally, as before we have
\[
V_{\psi^\full}^s(x/2)
\ll_s(\log\log{x})^3 D_k^\full(x)+\frac{x}{\log{x}},
\]
and this completes the proof.

\thispagestyle{empty}
{
\bibliographystyle{alpha}
\bibliography{reference}
}
\end{document}